\documentclass{conm-p-l}

\usepackage{amsthm,amsfonts, cite}
\usepackage{amssymb,graphicx,color}
\usepackage[all]{xy}
\usepackage{mathpazo}
\usepackage{eucal}
\usepackage{enumerate}
\usepackage{hyperref}
\numberwithin{equation}{section}

\theoremstyle{plain}
\newtheorem{teo}{Theorem}[section]
\newtheorem*{teo*}{Theorem}

\newtheorem{cor}[teo]{Corollary}
\newtheorem*{cor*}{Corollary}
\newtheorem{lem}[teo]{Lemma}
\newtheorem*{lem*}{Lemma}
\newtheorem{prop}[teo]{Proposition}
\newtheorem*{prop*}{Proposition}

\theoremstyle{remark}
\newtheorem{obs}[teo]{Remark}
\newtheorem{ex}[teo]{Example}

\newcommand{\R}{\ensuremath{{\mathbb{R}}}}

\newcommand{\g}{\ensuremath{\mathtt{g}}}

\begin{document}

\title[Codazzi tensors in reductive homogeneous spaces]{Codazzi tensor fields in reductive homogeneous spaces}

\author[J. Marshall Reber]{James Marshall Reber}
\author[I. Terek]{Ivo Terek}
\address{Department of Mathematics, The Ohio State University, Columbus, OH 43210, USA}
\email{marshallreber.1@osu.edu}
\email{terekcouto.1@osu.edu}


\begin{abstract}
We extend the results about left-invariant Codazzi tensor fields on Lie groups equipped with left-invariant Riemannian metrics obtained by d'Atri in 1985 to the setting of reductive homogeneous spaces $G/H$, where the curvature of the canonical connection of second kind associated with the fixed reductive decomposition $\mathfrak{g} = \mathfrak{h}\oplus\mathfrak{m}$ enters the picture. In particular, we show that invariant Codazzi tensor fields on a naturally reductive homogeneous space are parallel.
\end{abstract}
\maketitle

\section*{Introduction}

Whenever $M$ is a smooth manifold equipped with a connection $\nabla$, a twice-covariant symmetric tensor field $A$ on $M$ is called a \emph{Codazzi tensor field} if ${\rm d}^\nabla A = 0$, where ${\rm d}^\nabla$ is the exterior derivative operator (defined with the aid of $\nabla$) acting on tensor bundles over $M$, and we regard $A$ as a $T^*M$-valued $1$-form. When $\nabla$ is torsionfree, $A$ is a Codazzi tensor field if and only if
\begin{equation}\label{defn:Codazzi}
  (\nabla_{\boldsymbol{X}}A)(\boldsymbol{Y},\boldsymbol{Z}) = (\nabla_{\boldsymbol{Y}}A)(\boldsymbol{X},\boldsymbol{Z}),\quad\mbox{for all }\boldsymbol{X},\boldsymbol{Y},\boldsymbol{Z}\in \mathfrak{X}(M), \tag{$\dagger$}
\end{equation}which is to say that the covariant differential $\nabla A$, a three-times covariant tensor field on $M$, is totally symmetric.


Codazzi tensors are ubiquitous in geometry, with the most prominent examples being the second fundamental form of a non-degenerate hypersurface in a pseudo-Riemannian manifold with constant sectional curvature (due to the Codazzi-Mainardi compatibility equation), and the Ricci or Schouten tensors of a pseudo-Riemannian manifold with harmonic curvature or harmonic Weyl curvature (due to the relations ${\rm div}\,R = {\rm d}^\nabla {\rm Ric}$ and ${\rm div}\,{\rm W} = {\rm d}^\nabla{\rm Sch}$). Whenever a Riemannian manifold $(M,\g)$ has constant sectional curvature $K$, every Codazzi tensor field locally has the form ${\rm Hess}\,f + Kf\g$ for some smooth function $f$, cf. \cite{discussions}.

Both topological and geometric consequences of the existence of a nontrivial Codazzi tensor field on a Riemannian manifold have been studied in \cite{discussions, Derdzinski-Shen}, and the local structure of a Riemannian manifold carrying a Codazzi tensor field satisfying additional multiplicity assumptions on its spectra and eigendistributions is obtained in \cite{Gebarowski}. Many such results are compiled in \cite[\S 16.6--\S 16.22]{Besse}, which then led to further work \cite{Merton, Catino}.

In a different and more specific direction, left-invariant Codazzi tensor fields on Lie groups equipped with left-invariant Riemannian metrics have been discussed in \cite{dAtri}, with the goal of better understanding the harmonic curvature condition in this setting. New results have been recently obtained in \cite{Boucetta}, where it is shown that solvable Lie groups equipped with left-invariant Riemannian metrics having harmonic curvature must necessarily be Ricci-parallel.

In this paper, we extend the results in \cite{dAtri} to the more general class of invariant Codazzi tensor fields on reductive homogeneous spaces equipped with invariant Riemannian metrics. Our approach to achieve this is straightforward: once a reductive decomposition $\mathfrak{g} = \mathfrak{h}\oplus \mathfrak{m}$ for the homogeneous space $G/H$ is fixed, we run the computations done in \cite{dAtri} in the reductive complement $\mathfrak{m}$ (a non-associative algebra) instead of in the Lie algebra $\mathfrak{g}$. However, unlike in some results in \cite{dAtri} which involve positivity and negativity of sectional and scalar curvatures, the curvatures of $(G/H,\langle\cdot,\cdot\rangle)$ are now compared with curvatures of the canonical connection of second kind associated with the decomposition $\mathfrak{g} = \mathfrak{h}\oplus \mathfrak{m}$ --- with its flatness when $\mathfrak{h} = \{0\}$ and $\mathfrak{m} = \mathfrak{g}$ explaining its absence in \cite{dAtri}. Full proofs are included for the sake of completeness.

\section*{Organization of the text}

We work in the smooth category and all manifolds considered are connected.

In Section \ref{sec:preliminaries}, we gather some well-known standard facts regarding reductive homogeneous spaces needed for the rest of the text, the most important ones being Nomizu's Theorem \cite{Nomizu} on invariant connections and Lemma \ref{lem:cov_der_m}. Section \ref{sec:codazzi_comp} generalizes \cite[Proposition 1]{dAtri} to Proposition \ref{prop:comp_condition}: the same compatibility condition \eqref{eqn:compatibility} ensures that a symmetric bilinear form on $\mathfrak{m}$ reconstructed from prescribed eigenspaces gives rise to a Codazzi tensor field on $G/H$.

Section \ref{sec:versus_diff_curvatures} explores the effects of the existence of an invariant Codazzi tensor field on curvature, generalizing \cite[Propositions 3 and 4]{dAtri} and expressing the new conclusions, Propositions \ref{prop:sec_pos_neg} and \ref{prop:ricci_s}, with the aid of the \emph{difference curvature tensor} introduced in \eqref{eqn:difference_curvature}. In particular, we conclude that every invariant Codazzi tensor field on a naturally reductive homogeneous space is parallel.

\medskip

\noindent {\bf Acknowledgements.} We would like to thank Andrzej Derdzinski for bringing \cite{dAtri} to our attention, which ultimately motivated this work. We are also grateful to the anonymous referee, for the comments which allowed us to improve the exposition of the text.

\section{Preliminaries}\label{sec:preliminaries}

The material in this section is standard and it is included for the convenience of the reader. We refer to \cite[Ch. X]{KN2}, \cite[Ch. II]{Helgason}, and \cite[Ch. II--III]{Arvanitoyeorgos} for more details.

Let $G$ be a Lie group and $H$ be a closed Lie subgroup of $G$, so that the quotient space $G/H$ admits a unique smooth structure for which the natural projection $\pi\colon G \to G/H$ is a principal $H$-bundle. The group $G$ acts transitively on $G/H$ via the ``left translations'' $\tau_g\colon G/H \to G/H$ given by $\tau_g(aH)= (ga)H$. Writing $\mathfrak{g}$ and $\mathfrak{h}$ for the Lie algebras of $G$ and $H$, we assume that $G/H$ is \emph{reductive}: there is a \emph{vector space} direct sum decomposition $\mathfrak{g} = \mathfrak{h}\oplus \mathfrak{m}$ such that $\mathfrak{m}$ is \hbox{${\rm Ad}(H)$-invariant}. We write $(\cdot)_{\mathfrak{h}}\colon \mathfrak{g}\to \mathfrak{h}$ and $(\cdot)_{\mathfrak{m}}\colon \mathfrak{g}\to \mathfrak{m}$ for the direct sum projections, and so $(\mathfrak{m},[\cdot,\cdot]_{\mathfrak{m}})$ becomes a non-associative algebra. The derivative ${\rm d}\pi_e$ restricts to an isomorphism $\mathfrak{m}\cong T_{eH}(G/H)$ and, in addition,
\begin{equation}\label{eq:derivative_ad}
  \parbox{.68\textwidth}{for each $h\in H$, the derivative of $\tau_h\colon G/H\to G/H$ at the fixed point $eH$ is nothing more than ${\rm Ad}(h)\colon \mathfrak{m}\to\mathfrak{m}$.}
\end{equation}Our guiding principle is that for any $G$-equivariant smooth fiber bundle \hbox{$E \to G/H$},
\begin{equation}\label{eqn:equivariant_sections}
  \parbox{.57\textwidth}{$G$-equivariant sections of $E$ are in one-to-one correspondence with points of $E_{eH}$ fixed by $H$.}
\end{equation}Indeed, any point $\phi \in E_{eH}$ which is fixed by $H$ defines a $G$-equivariant section $\psi$ of $E$ via $\psi_{gH} = g\cdot \phi$. For example, taking $E$ to be tensor powers of $T^*(G/H)$ gives us that $G$-invariant covariant tensor fields on $G/H$ are in one-to-one correspondence with \hbox{${\rm Ad}(H)$-in\-va\-ri\-ant} covariant tensors on $\mathfrak{m}$, cf. \cite[Proposition 5.1]{Arvanitoyeorgos}, while taking $E$ to be Grassmannian bundles over $G/H$ yields that $G$-invariant distributions on $G/H$ are in one-to-one correspondence with \hbox{${\rm Ad}(H)$-invariant} vector subspaces of $\mathfrak{m}$. In addition, it has been proved in \cite{Tondeur} that
\begin{equation}\label{eqn:Tondeur}
  \parbox{.625\textwidth}{a $G$-invariant distribution $\mathcal{P}$ on $G/H$ is involutive if and only if the subspace $\,\mathcal{P}_{eH}\,$ is closed under $\,[\cdot,\cdot]_{\mathfrak{m}}$.}
\end{equation}
We will also need Nomizu's theorem \cite[Theorem 8.1]{Nomizu}:
\begin{equation}\label{eqn:nomizu}
  \parbox{.78\textwidth}{$G$-invariant affine connections on $G/H$ are in one-to-one correspondence with ${\rm Ad}(H)$-equivariant multiplications $\mathfrak{m}\times \mathfrak{m}\to \mathfrak{m}$.}
\end{equation}Following \cite[Section 5.2]{Elduque}, a $G$-invariant connection $\nabla$ on $G/H$ and an \linebreak[4]\hbox{${\rm Ad}(H)$-equivariant} multiplication $\alpha$ in $\mathfrak{m}$ related via \eqref{eqn:nomizu} determine each other by the relation \begin{equation}\label{eqn:alpha_and_nabla}
  \alpha(X,Y) = (\nabla_{X^{\#}}Y^{\#})|_{eH} + [X,Y]_{\mathfrak{m}},\quad\mbox{ for all }X,Y\in\mathfrak{m}.
\end{equation}Here, we are using that every $X\in\mathfrak{g}$ determines its corresponding action field $X^{\#} \in \mathfrak{X}(G/H)$, with $X^{\#}_{eH} = X_{\mathfrak{m}}$ and whose complete flow is explicitly given by $(t,aH)\mapsto \tau_{\exp(tX)}(aH)$. Note that the right-invariant vector field on $G$ generated by $X$ is $\pi$-related to $X^{\#}$. For future reference, we also observe that this implies that
\begin{equation}\label{eqn:Lie-der_G-inv}
  \parbox{.82\textwidth}{$\mathcal{L}_{X^{\#}}\varTheta = 0\,$ for every $\,X\in\mathfrak{g}\,$ and \hbox{$G$-invariant} tensor field $\,\varTheta\,$ on $G/H$,}
\end{equation}as the flow of $X^{\#}$ leaves $\varTheta$ invariant. The torsion and curvature of $\nabla$ are given in $\mathfrak{m}$ in terms of $\alpha$ by
\begin{equation}\label{eqn:geom_alpha}
  \parbox{.91\textwidth}{
    \begin{enumerate}[i)]
    \item $T(X,Y) = \alpha(X,Y)-\alpha(Y,X) - [X,Y]_{\mathfrak{m}}$,
    \item \scalebox{.99}{$R(X,Y)Z = \alpha(X,\alpha(Y,Z)) - \alpha(Y,\alpha(X,Z)) - \alpha([X,Y]_{\mathfrak{m}},Z) - [[X,Y]_{\mathfrak{h}}, Z],$}
    \end{enumerate}
  }\end{equation}for all $X,Y,Z\in\mathfrak{m}$, cf. \cite[formulas (9.1) and (9.6)]{Nomizu} or \cite[formula (22)]{Elduque}.

\begin{lem}\label{lem:cov_der_m}
For a $G$-invariant connection $\nabla$ and a $G$-invariant \hbox{$k$-times} covariant tensor field $\varTheta$ on $G/H$, corresponding to $\alpha$ and $\theta$ on $\mathfrak{m}$ under \eqref{eqn:nomizu}--\eqref{eqn:alpha_and_nabla} and \eqref{eqn:equivariant_sections}, the covariant differential $\nabla\varTheta$ is also \hbox{$G$-invariant} and corresponds under \eqref{eqn:equivariant_sections} to $\alpha(\cdot,\theta)$ on $\mathfrak{m}$ given by
\begin{equation}\label{eqn:cov_der_corr}
  \alpha(X,\theta)(Y_1,\ldots, Y_k) = -\sum_{i=1}^k \theta(Y_1,\ldots, \alpha(X,Y_i),\ldots, Y_k)
\end{equation}
for all $X,Y_1,\ldots, Y_k\in\mathfrak{m}$.
\end{lem}

\begin{proof}
  We will establish \eqref{eqn:cov_der_corr} when $k=1$, with the general case being an exercise in notation. The identity $(\nabla_{\boldsymbol{X}}\varTheta)(\boldsymbol{Y}) = (\mathcal{L}_{\boldsymbol{X}}\varTheta)(\boldsymbol{Y}) - \varTheta(\nabla_{\boldsymbol{X}}\boldsymbol{Y} - [\boldsymbol{X},\boldsymbol{Y}])$ evaluated at the vector fields $\boldsymbol{X} = X^{\#}$ and $\boldsymbol{Y} = Y^{\#}$, with $X,Y\in\mathfrak{m}$, reads as $(\nabla_{X^{\#}}\varTheta)(Y^{\#}) = -\varTheta(\nabla_{X^{\#}}Y^{\#} - [X^{\#},Y^{\#}])$ due to \eqref{eqn:Lie-der_G-inv}. As evaluating the relation $[X^{\#},Y^{\#}] = -[X,Y]^{\#}$ at $eH$ yields $[X^{\#},Y^{\#}]_{eH} = -[X,Y]_{\mathfrak{m}}$, \eqref{eqn:cov_der_corr} follows from \eqref{eqn:alpha_and_nabla}.
\end{proof}


Lastly, whenever $G/H$ is equipped with a $G$-invariant pseudo-Riemannian metric $\langle\cdot,\cdot\rangle$, $\alpha$ corresponding to the Levi-Civita connection under \eqref{eqn:nomizu}--\eqref{eqn:alpha_and_nabla} is called the \emph{Levi-Civita product} of $\langle\cdot,\cdot\rangle$. The Koszul formula for $\alpha$ becomes
\begin{equation}\label{eqn:koszul-alpha}
  2\big\langle \alpha(X,Y),Z\big\rangle = \big\langle [X,Y]_{\mathfrak{m}},Z\big\rangle - \big\langle X, [Y,Z]_{\mathfrak{m}}\big\rangle - \big\langle [X,Z]_{\mathfrak{m}},Y\big\rangle,
\end{equation}for all $X,Y,Z\in\mathfrak{m}$, cf. \cite[Exercise 10]{Elduque}.

\section{The Codazzi compatibility condition in $\mathfrak{m}$}\label{sec:codazzi_comp}

In this section, let $G/H$ be a homogeneous space admitting a reductive decomposition $\mathfrak{g} = \mathfrak{h}\oplus\mathfrak{m}$ be equipped with a $G$-invariant Riemannian metric $\langle\cdot,\cdot\rangle$ and its Levi-Civita product $\alpha$. By Lemma \ref{lem:cov_der_m} and \eqref{defn:Codazzi} in the Introduction, a twice-covariant $G$-invariant symmetric tensor field $A$ on $G/H$ is Codazzi if and only if
\begin{equation}\label{eqn:Cod_in_m}
  \alpha(X,A)(Y,Z) = \alpha(Y,A)(X,Z)
\end{equation}
for all $X,Y,Z\in\mathfrak{m}$. As $A$ is symmetric and $\g$ is positive-definite, the spectral theorem allows us to write an orthogonal direct sum decomposition
\begin{equation}\label{eqn:decomp}
\parbox{.795\textwidth}{$\mathfrak{m} = \mathfrak{m}_1\oplus \cdots \oplus \mathfrak{m}_r$, where $r\geq 1$ and each $\mathfrak{m}_i$ is the eigenspace of $A$ associated with the eigenvalue $\lambda_i$, ordered so that $\lambda_1 < \cdots < \lambda_r$.}
\end{equation}We will also write $(\cdot)_i\colon \mathfrak{m} \to \mathfrak{m}_i$ for the corresponding direct sum projections.

A subalgebra of $(\mathfrak{m},[\cdot,\cdot]_{\mathfrak{m}})$ is called \emph{totally geodesic} if it is closed under $\alpha$. By \eqref{eqn:Tondeur} and \eqref{eqn:alpha_and_nabla}, an ${\rm Ad}(H)$-invariant totally geodesic subalgebra of $(\mathfrak{m},[\cdot,\cdot]_{\mathfrak{m}})$ determines a foliation of $G/H$ by totally geodesic submanifolds. The next result generalizes \cite[Proposition 1]{dAtri}.

\begin{prop}\label{prop:comp_condition}
  Whenever $A$ is a $G$-invariant Codazzi tensor field on $G/H$, all the factors in decomposition \eqref{eqn:decomp} are ${\rm Ad}(H)$-invariant totally geodesic subalgebras of $(\mathfrak{m},[\cdot,\cdot]_{\mathfrak{m}})$, and the compatibility condition   \begin{equation}\label{eqn:compatibility}
    (\lambda_i-\lambda_k)^2 \big\langle [X_i,Y_j]_{\mathfrak{m}},Z_k\big\rangle + (\lambda_j-\lambda_i)^2\big\langle [X_i,Z_k]_{\mathfrak{m}},Y_j\big\rangle = 0
  \end{equation}
holds for all $X,Y,Z\in\mathfrak{m}$ and $i,j,k\in \{1,\ldots,r\}$. Conversely, if a direct sum decomposition $\mathfrak{m} =  \mathfrak{m}_1\oplus\cdots\oplus\mathfrak{m}_r$ into mutually orthogonal ${\rm Ad}(H)$-invariant vector subspaces is given and \eqref{eqn:compatibility} holds, any choice of mutually distinct real constants $\lambda_1,\ldots, \lambda_r$ gives rise to a \hbox{$G$-invariant} Codazzi tensor field on $G/H$ via $A = \bigoplus_{i=1}^n\lambda_i \langle\cdot,\cdot\rangle|_{\mathfrak{m}_i\times\mathfrak{m}_i}$. In addition, $\nabla A \neq 0$ if and only if there exists a triple $(i,j,k)$ of mutually distinct indices with $\langle X_i,[Y_j,Z_k]_{\mathfrak{m}}\rangle \neq 0$, in which case $A$ has at least three distinct eigenvalues.
\end{prop}

\begin{proof}
That each $\mathfrak{m}_i$ is ${\rm Ad}(H)$-invariant follows from ${\rm Ad}(H)$-invariance of both $A$ and $\langle\cdot,\cdot\rangle$. Namely, if $X\in\mathfrak{m}_i$, $h\in H$, and $Y\in\mathfrak{m}$, we have
\[ A({\rm Ad}(h)X,Y) = A(X,{\rm Ad}(h^{-1})Y) = \lambda_i\langle X, {\rm Ad}(h^{-1})Y\rangle = \lambda_i\langle {\rm Ad}(h)X,Y\rangle,
\]so that ${\rm Ad}(h)X\in\mathfrak{m}_i$. Next, as \eqref{eqn:koszul-alpha} is manifestly skew-symmetric in the pair $(Y,Z)$, we see that $\alpha(X,\cdot) \in \mathfrak{so}(\mathfrak{m},\langle\cdot,\cdot\rangle)$ for every $X\in\mathfrak{m}$, from which the relation \begin{equation}\label{eqn:pre-cod_eigenv} -\alpha(Z_k,A)(X_i,Y_j) = (\lambda_i-\lambda_j)\big\langle X_i,\alpha(Z_k,Y_j)\big\rangle\end{equation} follows for all $X,Y,Z\in\mathfrak{m}$. The Codazzi condition \eqref{eqn:Cod_in_m} now reads
  \begin{equation}\label{eqn:cod_eigenv}
    (\lambda_i-\lambda_j)\big\langle X_i, \alpha(Z_k,Y_j)\big\rangle = (\lambda_k-\lambda_j)\big\langle Z_k,\alpha(X_i,Y_j)\big\rangle.
  \end{equation}Using \eqref{eqn:koszul-alpha} twice and rearranging terms, \eqref{eqn:cod_eigenv} becomes
  \begin{equation}\label{eqn:technical}
    \begin{split}
      (\lambda_i-\lambda_k)\big\langle &[X_i,Y_j]_{\mathfrak{m}},Z_k\big\rangle +\\
      &+ (\lambda_i-\lambda_k)\big\langle [Z_k,Y_j]_{\mathfrak{m}},X_i\big\rangle + (\lambda_i+\lambda_k-2\lambda_j)\big\langle [X_i,Z_k]_{\mathfrak{m}},Y_j\big\rangle = 0.
    \end{split}
  \end{equation}Permuting elements, we also have
  \begin{equation}\label{eqn:technical_perm}
    \begin{split}
      (\lambda_j-\lambda_i)\big\langle &[Y_j,Z_k]_{\mathfrak{m}},X_i\big\rangle+ \\
      &+ (\lambda_j-\lambda_i)\big\langle [X_i,Z_k]_{\mathfrak{m}},Y_j\big\rangle + (\lambda_j+\lambda_i-2\lambda_k)\big\langle [Y_j,X_i]_{\mathfrak{m}},Z_k\big\rangle = 0,
    \end{split}\end{equation}and so $(\lambda_j-\lambda_i)\eqref{eqn:technical} + (\lambda_i-\lambda_k)\eqref{eqn:technical_perm} = 0$ becomes precisely \eqref{eqn:compatibility}. Making \hbox{$i=j\neq k$} on \eqref{eqn:compatibility} leads to $[X_i,Y_i]_{\mathfrak{m}} \in \mathfrak{m}_k^\perp$ for all $k\neq i$, so that $[X_i,Y_i]_{\mathfrak{m}}\in\mathfrak{m}_i$. Then, making $j=k\neq i$ on \eqref{eqn:compatibility} gives us that $\big\langle [X_i,Y_j]_{\mathfrak{m}},Z_j\big\rangle + \big\langle [X_i,Z_j]_{\mathfrak{m}},Y_j\big\rangle = 0$, which combined with \eqref{eqn:koszul-alpha} implies that each $\mathfrak{m}_i$ is closed under $\alpha$.

  Conversely, to verify that $A = \bigoplus_{i=1}^r\lambda_i\langle\cdot,\cdot\rangle|_{\mathfrak{m}_i\times\mathfrak{m}_i}$ defines a Codazzi tensor field whenever \eqref{eqn:compatibility} holds, it suffices to note that it implies \eqref{eqn:technical} (and hence \eqref{eqn:cod_eigenv}, due to \eqref{eqn:koszul-alpha}). Indeed: \eqref{eqn:compatibility} becomes \eqref{eqn:technical} when $i=k\neq j$ while, if $i\neq j$, adding to \eqref{eqn:compatibility} the expression obtained from it after permuting $(i,j,k)\mapsto (j,k,i)$ yields \eqref{eqn:technical_perm} (and hence \eqref{eqn:technical}).

  Finally, \eqref{eqn:compatibility} also implies
  \begin{equation}\label{eqn:intermediateDA}
    \parbox{.9\textwidth}{
      \begin{enumerate}[i)]
      \item $\displaystyle{\big\langle [X_i,Z_k]_{\mathfrak{m}},Y_j\big\rangle = -\frac{(\lambda_i-\lambda_k)^2}{(\lambda_j-\lambda_i)^2}\langle [X_i,Y_j]_{\mathfrak{m}},Z_k\rangle,}$
      \item $\displaystyle{\big\langle X_i,[Y_j,Z_k]_{\mathfrak{m}}\big\rangle = \frac{(\lambda_j-\lambda_k)^2}{(\lambda_j-\lambda_i)^2}\big\langle [X_i,Y_j]_{\mathfrak{m}},Z_k\big\rangle,}$
      \end{enumerate}
}
  \end{equation}whenever $i\neq j$. Substituting \eqref{eqn:intermediateDA} into \eqref{eqn:koszul-alpha} and simplifying it with the aid of \eqref{eqn:pre-cod_eigenv}, we obtain
  \begin{equation}\label{eqn:eigen_alpha}
    \big\langle \alpha(X_i,Y_j),Z_k\rangle = \frac{\lambda_i-\lambda_k}{\lambda_i-\lambda_j} \big\langle [X_i,Y_j]_{\mathfrak{m}},Z_k\big\rangle,\qquad i\neq j,
  \end{equation}which directly implies the last assertions regarding $\nabla A$.
\end{proof}

\begin{obs}
  The use of the spectral theorem to obtain \eqref{eqn:decomp} relies crucially on positive-definiteness of the Riemannian metric $\langle\cdot,\cdot\rangle$. When $\langle\cdot,\cdot\rangle$ has indefinite metric signature, we have \emph{Milnor's indefinite spectral theorem} \cite[p. 256]{Greub}:
\begin{equation}\label{eqn:Milnor_spectral}
  \parbox{.82\textwidth}{if a self-adjoint endomorphism $T$ of a pseudo-Euclidean space $(V,\langle\cdot,\cdot\rangle)$ with $\dim V \geq 3$ satisfies that $\langle Tv,v\rangle \neq 0$ for every null \hbox{$v\in V\smallsetminus \{0\}$}, then $T$ is diagonalizable in an orthonormal basis of $V$.}
\end{equation}To justify \eqref{eqn:Milnor_spectral}, it suffices to choose $\varPhi = \langle T\cdot,\cdot\rangle$ and $\varPsi = \langle\cdot,\cdot\rangle$ in the notation of \cite[p. 256]{Greub}. With \eqref{eqn:Milnor_spectral} in place, we see that $A$ gives rise to \eqref{eqn:decomp} and satisfies \eqref{eqn:compatibility} even when $\langle\cdot,\cdot\rangle$ has indefinite metric signature, provided that $\dim \mathfrak{m}\geq 3$ and $A(X,X) \neq 0$ whenever $X\in\mathfrak{m}\smallsetminus \{0\}$ is null. On the other hand, that \eqref{eqn:decomp} and \eqref{eqn:compatibility} together give rise to $G$-invariant Codazzi tensor fields on $G/H$ remains true without any additional assumptions.
\end{obs}

As pointed out in \cite{dAtri}, there is a simple interpretation for the compatibility relation \eqref{eqn:compatibility}. For each $k\in \{1,\ldots, r\}$, considering the inner product $\langle\!\langle \cdot,\cdot\rangle\!\rangle_k$ on $\mathfrak{m}$ defined by\footnote{Beware of the typo in \cite[formula (7)]{dAtri}: the formula there has $\langle X,Y\rangle$ instead of $\langle X_j,Y_j\rangle$.}
\[
  \langle\!\langle X,Y\rangle\!\rangle_k = \sum_{j=1}^r (\delta_{jk}+(\lambda_j-\lambda_k)^2)\langle X_j,Y_j\rangle,\qquad X,Y\in\mathfrak{m},
\]it follows that $\langle\!\langle [Z_k,X]_{\mathfrak{m}}, Y\rangle\!\rangle_k + \langle\!\langle X,[Z_k,Y]_{\mathfrak{m}}\rangle\!\rangle_k = 0$ for all $Z \in \mathfrak{m}_k$ and $X,Y\in\mathfrak{m}_k^\perp$. Indeed, it suffices to apply \eqref{eqn:compatibility}, assuming that $X\in\mathfrak{m}_i$ and $Y\in\mathfrak{m}_j$ with $i,j\neq k$. This means that, writing ${\rm ad}_{\mathfrak{m}}(X)(Y) = [X,Y]_{\mathfrak{m}}$ for every $X,Y\in \mathfrak{m}$ and denoting by $\pi_k^\perp$ the projection of $\mathfrak{m}$ onto $\mathfrak{m}_k^\perp$, the composition $(\pi_k^\perp\circ {\rm ad}_{\mathfrak{m}})|_{\mathfrak{m}_k}$ is a representation of $\,\mathfrak{m}_k\,$ on $\,(\mathfrak{m}_k^\perp,\langle\!\langle \cdot,\cdot\rangle\!\rangle_k)\,$ by skew-adjoint operators. Here, the representation is a representation of the vector space $\mathfrak{m}_k$, not of the non-associative algebra $(\mathfrak{m}_k,[\cdot,\cdot]_k)$. As a consequence:
\begin{equation}\label{eqn:pure_img}
  \parbox{.63\textwidth}{for each $Z_k\in\mathfrak{m}_k$, the chararacteristic roots of the operator $\,\pi_k^\perp\circ {\rm ad}_{\mathfrak{m}}(Z_k)|_{\mathfrak{m}_k^\perp}$ are all purely imaginary.}
\end{equation}

Recall that a non-associative algebra $\mathfrak{a}$ is:
\begin{enumerate}[(a)]
\item \emph{nilpotent} \cite[p. 18]{schafer} if there is a positive integer $t$ such that the product of $t$ elements in $\mathfrak{a}$, no matter how associated, equals zero.
\item \emph{split-solvable} (cf. \cite[p. 21]{knapp}) if there is a sequence $\mathfrak{a} = \mathfrak{a}_0 \supseteq \cdots \supseteq \mathfrak{a}_p = 0$ of ideals of $\mathfrak{a}$ with $\dim(\mathfrak{a}_i/\mathfrak{a}_{i+1})=1$ for every $i=0,\ldots,p-1$.
\end{enumerate}

Following \cite{dAtri}, we call a $G$-invariant Codazzi tensor field $A$ on $G/H$ \emph{essential} if $\nabla A\neq 0$ and none of the ei\-gen\-spa\-ces $\mathfrak{m}_i$ is an ideal of $(\mathfrak{m},[\cdot,\cdot]_{\mathfrak{m}})$. Note that $\mathfrak{m}_k$ is an ideal of $(\mathfrak{m},[\cdot,\cdot]_{\mathfrak{m}})$ if and only if $\pi_k^\perp \circ {\rm ad}_{\mathfrak{m}}(Z_k)|_{\mathfrak{m}_k^\perp} = 0$ for every $Z_k\in\mathfrak{m}_k$. Using the above, we obtain:

\begin{prop}
  If $G/H$ has a $G$-invariant essential Codazzi tensor field $A$, then $(\mathfrak{m},[\cdot,\cdot]_{\mathfrak{m}})$ cannot be nilpotent or split-solvable.
\end{prop}

\begin{proof}

As in the Lie category, one may define a `Killing form' $\beta$ for $(\mathfrak{m},[\cdot,\cdot]_{\mathfrak{m}})$ via $\beta(X,Y) = {\rm tr}({\rm ad}_{\mathfrak{m}}(X)\circ {\rm ad}_{\mathfrak{m}}(Y))$ for all $X,Y\in\mathfrak{m}$. A direct computation shows that, for every $Z_k\in\mathfrak{m}_k$, the relation 
\begin{equation}\label{eqn:Killing}
\beta(Z_k,Z_k) = \beta_k(Z_k,Z_k) + {\rm tr}\left[(\pi_k^\perp\circ {\rm ad}_{\mathfrak{m}}(Z_k)|_{\mathfrak{m}_k^\perp})^2\right]
\end{equation}
holds, where $\beta_k$ stands for the Killing form of $(\mathfrak{m}_k,[\cdot,\cdot]_k)$.

  Let $Z_k\in\mathfrak{m}_k$ be arbitrary, and assume that $(\mathfrak{m},[\cdot,\cdot]_{\mathfrak{m}})$ is nilpotent. It follows that both operators ${\rm ad}_{\mathfrak{m}}(Z_k)$ and ${\rm ad}_{\mathfrak{m}}(Z_k)|_{\mathfrak{m}_k}$ are nilpotent, and so both $\beta(Z_k,Z_k)$ and $\beta_k(Z_k,Z_k)$ vanish. In particular, \eqref{eqn:Killing} leads to ${\rm tr}\left[(\pi_k^\perp\circ {\rm ad}_{\mathfrak{m}}(Z_k)|_{\mathfrak{m}_k^\perp})^2\right]=0$. Together with \eqref{eqn:pure_img}, this implies that $\pi_k^\perp \circ {\rm ad}_{\mathfrak{m}}(Z_k)|_{\mathfrak{m}_k^\perp} = 0$. 

Now, assume instead that $(\mathfrak{m},[\cdot,\cdot]_{\mathfrak{m}})$ is split-solvable. By \cite[Corollary 1.30]{knapp}, whose `necessity' implication does not rely on the Jacobi identity, the characteristic roots of each ${\rm ad}_{\mathfrak{m}}(Z_k)$, for $Z_k \in \mathfrak{m}_k$, are real. Combined with \eqref{eqn:pure_img}, it follows that $\pi_k^\perp \circ {\rm ad}_{\mathfrak{m}}(Z_k)|_{\mathfrak{m}_k^\perp} = 0$ yet again.
\end{proof}

\section{Codazzi tensors versus difference curvatures}\label{sec:versus_diff_curvatures}

In this section, we continue to work with a homogeneous space $G/H$ equipped with a reductive decomposition $\mathfrak{g} = \mathfrak{h}\oplus\mathfrak{m}$, $G$-invariant Riemannian metric $\langle\cdot,\cdot\rangle$, and Levi-Civita product $\alpha$.

We will also need the \emph{canonical connection of second kind} induced by given reductive decomposition, that is, the affine connection $\nabla^0$ on $G/H$ corresponding under \eqref{eqn:nomizu}--\eqref{eqn:alpha_and_nabla} to the zero product in $\mathfrak{m}$. By (\ref{eqn:geom_alpha}-ii), the curvature tensor $R^0$ of $\nabla^0$ is given simply by $R^0(X,Y)Z = -[[X,Y]_{\mathfrak{h}},Z]$, for all $X,Y,Z\in\mathfrak{m}$. It follows from the Jacobi identity
\[  \sum_{{\rm cyc}} [[X,Y]_{\mathfrak{h}},Z] + \sum_{{\rm cyc}} [[X,Y]_{\mathfrak{m}},Z]_{\mathfrak{m}} = 0,\qquad X,Y,Z\in\mathfrak{m},   \]
and ${\rm Ad}(H)$-invariance of $\langle\cdot,\cdot\rangle$ that:
\begin{enumerate}[i)]
\item $(\mathfrak{m},[\cdot,\cdot]_{\mathfrak{m}})$ is a Lie algebra if and only if $R^0$ satisfies the Bianchi identity,
    \item the expression $\,\langle R^0(X,Y)Z,W\rangle\,$ is skew-symmetric in the pair $(Z,W)$.
\end{enumerate}

The \emph{Ricci tensor} ${\rm Ric}^0$ of $\nabla^0$ is defined by ${\rm Ric}^0(Y,Z) = {\rm tr}(X\mapsto R^0(X,Y)Z)$, with no reference to the metric $\langle\cdot,\cdot\rangle$, and it is only guaranteed to be symmetric if $R^0$ satisfies the Bianchi identity. We also consider the \emph{sectional} and \emph{scalar curvature functions} $K^0$ and ${\rm s}^0$ associated with $\nabla^0$ and $\langle\cdot,\cdot\rangle$: for any plane $\Pi\subseteq \mathfrak{m}$ we let $K^0(\Pi) = \langle R^0(X,Y)Y,X\rangle$, where $\{X,Y\}$ is any orthonormal basis for $\Pi$ (with its choice being immaterial due to (ii) above), and ${\rm s}^0 = {\rm tr}_{\langle\cdot,\cdot\rangle}\,{\rm Ric}^0$.

The results in this section are most conveniently stated and proved in terms of
\begin{equation}\label{eqn:difference_curvature}
  \parbox{.82\textwidth}{the \emph{difference curvature tensor} $R^d = R-R^0$ and the corresponding notions of sectional, Ricci, and scalar curvatures: they are respectively defined by $K^d = K-K^0\,$, $\,{\rm Ric}^d = {\rm Ric} - {\rm Ric}^0\,$, and $\,{\rm s}^d = {\rm s} - {\rm s}^0$.}
\end{equation}

As setup for the next result, observe that whenever $A$ is a $G$-invariant Codazzi tensor field on $G/H$ and $\mathfrak{m}$ is decomposed as in \eqref{eqn:decomp}, an equivalent formulation to \eqref{eqn:eigen_alpha} is
  \begin{equation}\label{eqn:eigen_alpha_2}
    \alpha(X_i,Y_j) =\sum_{k=1}^r \frac{\lambda_i-\lambda_k}{\lambda_i-\lambda_j} [X_i,Y_j]_k,\qquad i\neq j.
  \end{equation}
  Applying \eqref{eqn:eigen_alpha_2} to separately compute each term in the curvature relation (\ref{eqn:geom_alpha}-ii) for $(X,Y,Z) = (X_i,Y_j,Y_j)$, with $i\neq j$, we obtain $\big\langle \alpha(X_i,\alpha(Y_j,Y_j)),X_i\big\rangle = 0$ and
\begin{equation}\label{eqn:intermediate_terms_sect}
\scalebox{.97}{$\displaystyle{\big\langle \alpha(Y_j, \alpha(X_i,Y_j)), X_i\big\rangle = \big\langle \alpha([X_i,Y_j]_{\mathfrak{m}},Y_j),X_i\big\rangle  = \sum_{{k=1}\atop {k\neq j}}^r \frac{\lambda_k-\lambda_i}{\lambda_j-\lambda_k} \big\langle [Y_j,[X_i,Y_j]_k]_i,X_i\big\rangle}$}.
  \end{equation}
Choosing $Z=[X_i,Y_j]_k$ and switching the roles of $X$ and $Y$ in \eqref{eqn:compatibility} leads to
\[  -(\lambda_j-\lambda_k)^2\|[X_i,Y_j]_k\|^2 + (\lambda_j-\lambda_i)^2\big\langle [Y_j,[X_i,Y_j]_k]_i,X_i\big\rangle = 0\]which, when combined with \eqref{eqn:intermediate_terms_sect}, implies that
\begin{equation}\label{eqn:sec_d}
  \big\langle R^d(X_i,Y_j)Y_j,X_i\big\rangle = \frac{2}{(\lambda_i-\lambda_j)^2} \sum_{{k=1}\atop {k\neq j}}^r (\lambda_i-\lambda_k)(\lambda_j-\lambda_k) \|[X_i,Y_j]_k\|^2.
\end{equation}

We are ready to generalize \cite[Proposition 3]{dAtri}:

\begin{prop}\label{prop:sec_pos_neg}
If $G/H$ has a $G$-invariant Codazzi tensor field $A$ with $\nabla A \neq 0$, the difference sectional curvature $K^d$ assumes both positive and negative values.
\end{prop}

\begin{proof}
  We claim that
  \begin{equation}\label{eqn:not_subalgebras}  \parbox{.62\textwidth}{there is a smallest integer $2\leq \rho \leq r-1$, as well as integers $1 \leq \mu<\nu\leq r$, such that ${\rm (a)}$ $\mathfrak{m}_1\oplus\cdots\oplus \mathfrak{m}_\rho$ and ${\rm (b)}$ $\,\mathfrak{m}_\mu\oplus \mathfrak{m}_\nu\,$ are \emph{not} subalgebras of $\,(\mathfrak{m},[\cdot,\cdot]_{\mathfrak{m}})$.}  \end{equation}
  If either (\ref{eqn:not_subalgebras}-${\rm a}$) or (\ref{eqn:not_subalgebras}-${\rm b}$) fails to hold, then $\langle [X_i,Y_j]_{\mathfrak{m}},Z_k\rangle = 0$ whenever $i,j,k$ are mutually distinct, so that $\nabla A = 0$ by Proposition \ref{prop:comp_condition}. Indeed, if ${\rm (a)}$ fails then $\langle [X_i,Y_j]_{\mathfrak{m}},Z_k\rangle = 0$ whenever $k>\max\{i,j\}$ as $[\mathfrak{m}_i,\mathfrak{m}_j]_{\mathfrak{m}}\subseteq \mathfrak{m}_1\oplus\cdots\oplus\mathfrak{m}_{\max\{i,j\}}$ is orthogonal to $\mathfrak{m}_k$, and we may apply \eqref{eqn:compatibility}. If ${\rm (b)}$ fails instead, then again $[\mathfrak{m}_i,\mathfrak{m}_j]_{\mathfrak{m}} \subseteq \mathfrak{m}_i\oplus \mathfrak{m}_j$ is orthogonal to $\mathfrak{m}_k$ whenever $i,j,k$ are mutually distinct. This proves \eqref{eqn:not_subalgebras}.

 For $\rho$ as in (\ref{eqn:not_subalgebras}-${\rm a}$), minimality of $\rho$ implies that $[\mathfrak{m}_i,\mathfrak{m}_j]_\rho = 0$ whenever $i,j<\rho$, and so $[\mathfrak{m}_i,\mathfrak{m}_\rho]_j=0$ for distinct $i,j < \rho$ by \eqref{eqn:compatibility} with $k=\rho$. Hence, \eqref{eqn:decomp} and \eqref{eqn:sec_d} yield \[  K^d(\Pi) = \frac{2}{(\lambda_i-\lambda_\rho)^2}\sum_{k=\rho+1}^r (\lambda_i-\lambda_k)(\lambda_\rho-\lambda_k)\|[X_i,Y_\rho]_k\|^2 > 0   \]for $\Pi = \R X_i \oplus \R Y_\rho$ with $i<\rho$, $\|X_i\|=\|Y_\rho\|=1$, and $[X_i,Y_\rho]_{\mathfrak{m}}\neq 0$.

Lastly, for $\mu,\nu$ as in (\ref{eqn:not_subalgebras}-${\rm b}$) chosen so that the difference $\nu-\mu$ is maximal, we have that $\mathfrak{m}_i\oplus\mathfrak{m}_j$ is a subalgebra of $(\mathfrak{m},[\cdot,\cdot]_{\mathfrak{m}})$ for $1\leq i \leq \mu < \nu \leq j \leq r$, provided that $i\neq \mu$ or $j\neq \nu$. This implies that $[\mathfrak{m}_k,\mathfrak{m}_\mu]_{\nu} = [\mathfrak{m}_\nu,\mathfrak{m}_k]_{\mu} = 0$ whenever $k<\mu$ or $k > \nu$, and thus $[\mathfrak{m}_\mu,\mathfrak{m}_\nu]_k=0$ by \eqref{eqn:compatibility} with $(\mu,\nu) = (i,j)$. Choosing unit vectors $X_\mu$ and $Y_\nu$ with $[X_\mu,Y_\nu]_\ell\neq 0$, for some $\ell\neq \mu,\nu$, it follows from \eqref{eqn:decomp} and \eqref{eqn:sec_d} that \[  K^d(\Pi) = \frac{2}{(\lambda_\mu-\lambda_\nu)^2}\sum_{k=\mu}^\nu (\lambda_\mu-\lambda_k)(\lambda_\nu-\lambda_k)\|[X_\mu,Y_\nu]_k\|^2 < 0  \]for $\Pi = \R X_\mu\oplus \R Y_\nu$, as required.
\end{proof}

\begin{ex}
Recall that a homogeneous space $G/H$ with a $G$-invariant Riemannian metric $\langle\cdot,\cdot\rangle$ is called \emph{naturally reductive} if it admits a reductive decomposition $\mathfrak{g} = \mathfrak{h}\oplus\mathfrak{m}$ with the additional property that $\langle [X,Y]_{\mathfrak{m}},Z\rangle + \langle Y, [X,Z]_{\mathfrak{m}}\rangle = 0$, for all $X,Y,Z\in\mathfrak{m}$. Rearranging the formula in \cite[Proposition 5.7]{Arvanitoyeorgos} we see that, in this case, $K^d(\Pi) =  \|[X,Y]_{\mathfrak{m}}\|^2/4\geq 0$, where $\{X,Y\}$ is any orthonormal basis for $\Pi$. By Proposition \ref{prop:sec_pos_neg}, every $G$-invariant Codazzi tensor field on such a naturally reductive homogeneous space is necessarily parallel.
\end{ex}

For the next result, which generalizes \cite[Proposition 4]{dAtri}, we let $M_i$ be the leaf passing through $eH$ of the eigendistribution of $A$ associated with $\lambda_i$, so that $T_{eH}M_i = \mathfrak{m}_i$. Each $M_i$ is a totally geodesic submanifold of $G/H$ equipped either with the Levi-Civita connection of $\langle\cdot,\cdot\rangle$ (by Proposition \ref{prop:comp_condition}), or with the canonical connection $\nabla^0$. This allows us to consider the difference Ricci and scalar curvatures ${\rm Ric}^d_i$ and ${\rm s}^d_i$ in \eqref{eqn:difference_curvature} for each $M_i$. More precisely, given $Y_i,Z_i\in \mathfrak{m}_i$, the endomorphism $X\mapsto R^d(X,Y_i)Z_i$ of $\mathfrak{m}$ restricts to an endomorphism of $\mathfrak{m}_i$, whose trace is ${\rm Ric}_i^d(Y_i,Z_i)$. Then, the trace of ${\rm Ric}_i^d$ computed with $\langle\cdot,\cdot\rangle|_{\mathfrak{m}_i\times\mathfrak{m}_i}$ is ${\rm s}^d_i$.

\begin{prop}\label{prop:ricci_s}
  If $G/H$ has a $G$-invariant Codazzi tensor field, then:
  \begin{enumerate}[\normalfont i)]
  \item ${\rm Ric}^d(Y_j,Y_j) \leq {\rm Ric}^d_j(Y_j,Y_j)\,$ for $j\in \{1,r\}$ and all $Y\in\mathfrak{m}$.
  \item ${\rm s}^d_1+\cdots+{\rm s}^d_r = {\rm s}^d$.
  \end{enumerate}
\end{prop}

\begin{proof}
  First, observe that the cyclic identity
  \begin{equation}\label{eqn:cyclic_identity}
  \begin{split}  \frac{(\lambda_i-\lambda_k)(\lambda_j-\lambda_k)}{(\lambda_i-\lambda_j)^2}\langle [X_i,Y_j]_{\mathfrak{m}},Z_k\rangle^2 &+   \frac{(\lambda_j-\lambda_i)(\lambda_k-\lambda_i)}{(\lambda_j-\lambda_k)^2}\langle [Y_j,Z_k]_{\mathfrak{m}},X_i\rangle^2 + \\ &+ \frac{(\lambda_k-\lambda_j)(\lambda_i-\lambda_j)}{(\lambda_k-\lambda_i)^2}\langle [Z_k,X_i]_{\mathfrak{m}},Y_j\rangle^2=0 \end{split}
  \end{equation}
holds for all $X,Y,Z\in\mathfrak{m}$ whenever $i$, $j$ and $k$ are mutually distinct, as a direct consequence of \eqref{eqn:intermediateDA}. Now, writing $d_i = \dim \mathfrak{m}_i$ and letting $\{E_{i,a}\}_{a=1}^{d_i}$ be an orthonormal basis for $\mathfrak{m}_i$, for each $i=1,\ldots,r$, it follows from the definition of ${\rm Ric}_j^d$ and \eqref{eqn:sec_d} that
\begin{equation}\label{eqn:Ricci_first}
  {\rm Ric}^d(Y_j) = {\rm Ric}^d_j(Y_j) + 2\sum_{{i=1}\atop {i\neq j}}^r\sum_{a=1}^{d_i} \sum_{{k=1}\atop {k\neq j}}^r \sum_{b=1}^{d_k} \frac{(\lambda_i-\lambda_k)(\lambda_j-\lambda_k)}{(\lambda_i-\lambda_j)^2}\langle [E_{i,a},Y_j]_{\mathfrak{m}}, E_{k,b}\rangle^2
\end{equation}
for every $Y_j\in\mathfrak{m}_j$. Here, we write ${\rm Ric}^d(Y_j)$ as a shorthand for ${\rm Ric}^d(Y_j,Y_j)$, and similarly for ${\rm Ric}^d_j$. The summand in the right side of \eqref{eqn:Ricci_first} vanishes when $k=i$ and, relabeling dummy indices $(i,a) \leftrightharpoons (k,b)$ in one of the two copies of such summation, we see that \eqref{eqn:cyclic_identity} leads to
\begin{equation}\label{eqn:Ricci_estimate}
  {\rm Ric}^d(Y_j) = {\rm Ric}^d_j(Y_j) -\sum_{{i=1}\atop {i\neq j}}^r\sum_{a=1}^{d_i} \sum_{{k=1}\atop {k\neq j}}^r \sum_{b=1}^{d_k} \frac{(\lambda_k-\lambda_j)(\lambda_i-\lambda_j)}{(\lambda_k-\lambda_i)^2}\langle [E_{k,b},E_{i,a}]_{\mathfrak{m}},Y_j\rangle^2.
\end{equation}Using \eqref{eqn:decomp} and the fact that $(\lambda_k-\lambda_j)(\lambda_i-\lambda_j)$ is a product of positive (or, negative) factors when $j=1$ (or, $j=r$) for all $i$ and $k$, (i) follows. Finally, setting $Y_j = E_{j,c}$ in \eqref{eqn:Ricci_estimate} and summing over $1\leq c\leq d_j$ and $1\leq j\leq r$, we conclude that (ii) holds: the difference ${\rm s}_1^d+\cdots + {\rm s}_r^d - {\rm s}^d$ equals the sum over mutually distinct indices $i,j,k$ of terms appearing in \eqref{eqn:cyclic_identity}, and therefore it must vanish.

\end{proof}

A last consequence of Proposition \ref{prop:ricci_s} is the counterpart to \cite[Proposition 5]{dAtri}:

\begin{cor}
  Suppose that ${\rm Ric}^d$ itself is a Codazzi tensor field on $G/H$, with $\nabla {\rm Ric}^d \neq 0$. If ${\rm s}^d_i \geq 0$ for $1\leq i \leq r-1$, then ${\rm s}^d_r \neq 0$. In particular, not all eigenspaces of ${\rm Ric}^d$ can be Abelian subalgebras of $(\mathfrak{m},[\cdot,\cdot]_{\mathfrak{m}})$.
\end{cor}

\begin{proof}
  Item (i) of Proposition \ref{prop:ricci_s} for $A={\rm Ric}^d$ reads ${\rm Ric}^d(Y_r,Y_r) \geq \lambda_r$ for all unit vectors $Y_r\in\mathfrak{m}_r$, so averaging over an orthonormal basis yields ${\rm s}^d_r/d_r \geq \lambda_r$. If it were to be ${\rm s}^d_r=0$, \eqref{eqn:decomp} would imply that $\lambda_1<\cdots<\lambda_r\leq 0$, and hence ${\rm s}^d = d_1\lambda_1+\cdots + d_r\lambda_r<0$. However, it is clear from ${\rm s}^d_i \geq 0$, for $1\leq i \leq r-1$, and item (ii) of Proposition \ref{prop:ricci_s}, that ${\rm s}^d \geq 0$. The last claim now follows as $R^d_i = 0$ (and thus ${\rm s}^d_i=0$) whenever $\mathfrak{m}_i$ is Abelian, as $\alpha|_{\mathfrak{m}_i\times\mathfrak{m}_i}=0$ in view of \eqref{eqn:koszul-alpha} and Proposition \ref{prop:comp_condition}.
\end{proof}

\bibliography{codazzi_refs}{}

\begin{thebibliography}{10}

\bibitem{Boucetta}
I.~Aberaouze and M.~Boucetta.
\newblock Left invariant {R}iemannian metrics with harmonic curvature are
  {R}icci-parallel in solvable {L}ie groups and {L}ie groups of dimension
  {$\leq 6$}.
\newblock {\em J. Geom. Phys.}, 177:Paper No. 104517, 14, 2022.

\bibitem{Arvanitoyeorgos}
A.~Arvanitoyeorgos.
\newblock {\em An introduction to {L}ie groups and the geometry of homogeneous
  spaces}, volume~22 of {\em Student Mathematical Library}.
\newblock American Mathematical Society, Providence, RI, 2003.
\newblock Translated from the 1999 Greek original and revised by the author.

\bibitem{Besse}
A.~L. Besse.
\newblock {\em Einstein manifolds}.
\newblock Classics in Mathematics. Springer-Verlag, Berlin, 2008.
\newblock Reprint of the 1987 edition.

\bibitem{discussions}
I.~Bivens, J.-P. Bourguignon, A.~Derdzi\'{n}ski, D.~Ferus, O.~Kowalski,
  T.~Klotz Milnor, V.~Oliker, U.~Simon, W.~Str\"{u}bing, and K.~Voss.
\newblock Discussion on {C}odazzi-tensors.
\newblock In {\em Global differential geometry and global analysis ({B}erlin,
  1979)}, volume 838 of {\em Lecture Notes in Math.}, pages 243--299. Springer,
  Berlin, 1981.

\bibitem{Catino}
G.~Catino, C.~Mantegazza, and L.~Mazzieri.
\newblock A note on {C}odazzi tensors.
\newblock {\em Math. Ann.}, 362(1-2):629--638, 2015.

\bibitem{dAtri}
J.~E. D'Atri.
\newblock Codazzi tensors and harmonic curvature for left invariant metrics.
\newblock {\em Geom. Dedicata}, 19(3):229--236, 1985.

\bibitem{Derdzinski-Shen}
A.~Derdzi\'{n}ski and C.~L. Shen.
\newblock Codazzi tensor fields, curvature and {P}ontryagin forms.
\newblock {\em Proc. London Math. Soc. (3)}, 47(1):15--26, 1983.

\bibitem{Elduque}
A.~Elduque.
\newblock Reductive homogeneous spaces and nonassociative algebras.
\newblock {\em Commun. Math.}, 28(2):199--229, 2020.

\bibitem{Gebarowski}
A.~Gebarowski.
\newblock The structure of certain {R}iemannian manifolds admitting {C}odazzi
  tensors.
\newblock {\em Demonstratio Math.}, 27(1):249--252, 1994.

\bibitem{Greub}
W.~H. Greub.
\newblock {\em Linear algebra}.
\newblock Die Grundlehren der mathematischen Wissenschaften, Band 97.
  Springer-Verlag New York, Inc., New York, third edition, 1967.

\bibitem{Helgason}
S.~Helgason.
\newblock {\em Differential geometry, {L}ie groups, and symmetric spaces},
  volume~34 of {\em Graduate Studies in Mathematics}.
\newblock American Mathematical Society, Providence, RI, 2001.
\newblock Corrected reprint of the 1978 original.

\bibitem{knapp}
A.~W. Knapp.
\newblock {\em Lie Groups: Beyond an Introduction}, volume 140 of {\em Progress
  in Mathematics}.
\newblock Birkh\"{a}user Boston, Inc., Boston, MA, 1996.

\bibitem{KN2}
S.~Kobayashi and K.~Nomizu.
\newblock {\em Foundations of differential geometry. {V}ol. {II}}.
\newblock Wiley Classics Library. John Wiley \& Sons, Inc., New York, 1996.
\newblock Reprint of the 1969 original, A Wiley-Interscience Publication.

\bibitem{Merton}
G.~Merton.
\newblock Codazzi tensors with two eigenvalue functions.
\newblock {\em Proc. Amer. Math. Soc.}, 141(9):3265--3273, 2013.

\bibitem{Nomizu}
K.~Nomizu.
\newblock Invariant affine connections on homogeneous spaces.
\newblock {\em Amer. J. Math.}, 76:33--65, 1954.

\bibitem{schafer}
R.~D. Schafer.
\newblock {\em An Introduction to Nonassociative Algebras}.
\newblock Dover Publications, Inc., New York, 1995.
\newblock Corrected reprint of the 1966 original.

\bibitem{Tondeur}
P.~Tondeur.
\newblock Invariant subbundles of the tangentbundle of a reductive homogeneous
  space.
\newblock {\em Math. Z.}, 89:420--421, 1965.

\end{thebibliography}
\bibliographystyle{plain}

\end{document}